\documentclass[a4paper,11pt]{amsart}
\usepackage[english]{babel}
\usepackage{graphicx,color}
\usepackage{float}
\usepackage{amsmath}
\usepackage{bbold}
\usepackage{amssymb}
\usepackage{amsthm}
\usepackage{MnSymbol}
\usepackage{wasysym}
\usepackage{parskip}
\usepackage[letterpaper, margin=1in]{geometry}
\usepackage[document]{ragged2e}
\usepackage{personal}
\usepackage{todonotes}

\usepackage{tikz}
\usetikzlibrary{patterns}

\DeclareMathOperator*{\essinf}{ess\,inf}

\newcommand{\ex}{{\rm e}}

\newcommand{\eqn}[1]{\begin{equation}#1\end{equation}}
\newcommand{\eqan}[1]{\begin{align}#1\end{align}}

\newcommand{\op}{o_{\sss {\mathbb P}}}

\newcommand{\convas}{\stackrel{\sss a.s.}{\longrightarrow}}

\newcommand{\prob}{\mathbb P}
\newcommand{\expec}{\mathbb E}

\newtheorem{remark}[theorem]{Remark}

\def\1{{\mathchoice {1\mskip-4mu\mathrm l}      
{1\mskip-4mu\mathrm l}
{1\mskip-4.5mu\mathrm l} {1\mskip-5mu\mathrm l}}}
\newcommand{\indic}[1]{\1_{\{#1\}}}

\newcommand{\sss}{\scriptscriptstyle}

\newcommand{\nn}{\nonumber}

\title{Large Deviations of Bivariate Gaussian Extrema}
\begin{document}

\author{Remco van der Hofstad}\address{Eindhoven University of Technology, P.O. Box 513, 5600 MB Eindhoven, The Netherlands. E-mail: rhofstad@win.tue.nl}
\author{Harsha Honnappa}\address{Purdue University, 315 N. Grant St., West Lafayette IN 47906 U.S.A. Email: honnappa@purdue.edu}
\maketitle

\begin{abstract}
We establish sharp tail asymptotics for component-wise extreme values of bivariate Gaussian random vectors with arbitrary correlation between the components. We consider two scaling regimes for the tail event in which we demonstrate the existence of a restricted large deviations principle, and identify the unique rate function associated with these asymptotics. Our results identify when the maxima of both coordinates are typically attained by two different vs.\ the same index, and how this depends on the correlation between the coordinates of the bivariate Gaussian random vectors.

Our results complement a growing body of work on the extremes of Gaussian processes. The results are also relevant for steady-state performance and simulation analysis of networks of infinite server queues.
\end{abstract}

\section{Introduction}\justify
Motivated by applications to the analysis of queueing networks, we study the large deviations of extreme values of multivariate Gaussians. We focus on the bivariate case for simplicity, but our analysis will carry over to the more general case with some effort. Let $\{X_1, X_2,\ldots\}$ be an ensemble of independent and identically distributed (i.i.d.) bi-variate Gaussians with covariance matrix $\Sigma$,
	and let $\bar X_n := (\max_{1\leq i\leq n} X^{\sss(1)}_i, \max_{1\leq i\leq n}X_i^{\sss(2)})$ be the component-wise maximum, or extreme value, random vector. For simplicity we assume that $\bbE[X_1] = 0$. In the context of a queueing network $\bar X_n$ is an approximation to the maximum congestion experienced over a typical interval in a network of infinite server queues, for instance. We characterize the likelihood of the tail event $\{\bar X_n > a_n u\}$ where $u \in (0,\infty)$ as the typical interval $n$ tends to infinity, under the assumption that $a_n \to \infty$ as $n \to\infty$. We consider two cases. 
	
\underline{Case 1: The right scale.} Under the condition that $a_n = \sqrt{\log n}$, we prove a ``restricted'' large deviations principle (RLDP) (in the sense of~\cite{KoMa2015}) in Theorem~\ref{thm-main-small} that shows that if $u > \sqrt{2} (\sigma^{\sss(1)},\sigma^{\sss(2)})$ (where $\sigma^{\sss(j)}$ is the standard deviation of marginal $j$) then
	\begin{align}
	\label{right-scale}
		\lim_{n\to\infty} \frac{1}{a_n^2} \log \bbP(\bar X_n > a_n u) = J(u/\sigma),
	\end{align}
where
\begin{align}
\label{J(u)-def}
	J(u)= \begin{cases}
 		1 - \frac{1}{2} \left( {u^{\sss(1)}} \right)^2 &\text{when}~u^{\sss(2)} \leq \rho u^{\sss(1)},\\
		1 - \frac{1}{2} \left( {u^{\sss(2)}} \right)^2 &\text{when}~u^{\sss(1)} \leq \rho u^{\sss(2)},\\
 		\max\Big\{2 - \tfrac{1}{2} \left\|{u} \right\|_2^2, 1-\frac{(u^{\sss(1)}/\sigma^{\sss(1)})^2-2\rho (u^{\sss(1)}/\sigma^{\sss(1)}) (u^{\sss(2)}/\sigma^{\sss(2)})+(u^{\sss(2)}/\sigma^{\sss(2)})^2}{2(1-\rho^2)}\Big\}&\text{otherwise,}
 	\end{cases}
	\end{align}
 $u/\sigma := (u^{\sss(1)}/\sigma^{\sss(1)}, u^{\sss(2)}/\sigma^{\sss(2)})$, $\|u\|_2^2 := \sum_{i=1,2} |u^{\sss(i)}|^2$ and $\rho \in [-1,+1]$. Here, the top two cases only arise when $\rho>0$, so when $\rho<0$, the last line equals $J(u)$. The proof follows by using the Laplace principle in the key Lemma~1, and combining it with the `largest probability wins' principle. 

The different cases in \eqref{J(u)-def} originate due to the different scenarios in which the bivariate distribution can attain its maximum. In all the cases where a term +1 is present, the maximum is attained by {\em one} index of $X_i$ which simultaneously attains the maximum of both coordinates. In all the cases where a term +2 is present, the maximum is attained by {\em two different} indices of $X_i$, one which attains the maximum of the first coordinate, and one which attains the maximum of the second coordinate. The latter case has most distinct possibilities (``larger entropy''), while the first may have a larger probability for appropriate correlation coefficients $\rho$. The optimal strategy is characterized by the `largest probability wins' principle. 
 	
\underline{Case 2: Larger scales.} On a much larger scale, where $a_n \gg \sqrt{\log n}$, we establish two main results. First, we prove a leading order asymptote for the extreme value that aligns with the result in case 1. Precisely, in Theorem~\ref{thm-leading-order} we prove the RLDP 
	\eqn{
	\label{LD-large}
	\lim_{n\to\infty} \frac{1}{a_n^2} \log \prob\big(\bar X_n> a_n u) = -I(u/\sigma),
	}
where
	\eqn{
	\label{rate-BVN}
	I(u)=\begin{cases}
	\frac{1}{2}(u^{\sss(1)})^2 &\text{when }u^{\sss(2)}\leq \rho u^{\sss(1)},\\
	\frac{1}{2} \left( {u^{\sss(2)}} \right)^2 &\text{when}~u^{\sss(1)} \leq \rho u^{\sss(2)},\\
	\frac{1}{2}\min\big\{\|u\|^2_2, \frac{(u^{\sss(1)})^2-2\rho u^{\sss(1)} u^{\sss(2)}+(u^{\sss(2)})^2)}{1-\rho^2}\big\}
	&\text{otherwise.}
	\end{cases}
	}
%

In Theorem~\ref{thm-sharp-order}, on the other hand, we establish a sharp asymptote for the likelihood and show that there exists a continuous function $I \colon \bbR^d \to \bbR$ and constants $K$, $b$ and $c$ such that
	\begin{align}
	\label{Bahadur-Rao}
		\lim_{n\to\infty} a_n^b n^c {\mathrm e}^{a_n^2 I(u)} \bbP(\bar X_n > a_n u) = K.
	\end{align}
	The proofs of these theorems uses the inclusion-exclusion principle to bound the likelihood from above and below. 

	\paragraph{\underline{Related Literature}}
Multivariate Gaussians emerge as stationary limits of networks of $G/G/\infty$ infinite server queues when the arrival rate is high (i.e., in heavy-traffic in the sense of~\cite{glynn1991new}); see \cite{borovkov1967limit,whitt1982heavy} as well. This is straightforward to observe in the case of a single $M/G/\infty$ station where the number in system in steady-state is Poisson distributed. When the arrival rate is high the steady-state distribution is well approximated by a Gaussian.

Next, there is an explicit connection with extreme value theory (EVT). The logarithmic asymptotics established here complement the uniform convergence results for EVT; see \cite[Chapter 4]{Vi1994} and \cite[Section I, Chapter C]{de1983}.
There are also clear connections with recent work on extremes of multidimensional Gaussian processes in \cite{DeHaJiRo2018,DeHaJiTa2015,DeKoMaRo2010,KoMa2015,honnappa2018dominating} and other related work, where logarithmic asymptotics are derived for the ``at least one in the set" extremum (not the component-wise extremum considered here) for Gaussian processes. We note, in particular, ~\cite{KoMa2015} where logarithmic asymptotics are derived for the ``at least one in set" extremum of a sequence of (non-i.i.d.) generally distributed random vectors. The authors present a general theory closely aligned with the RLDP for univariate random variables introduced in~\cite{DuLeSu2003}, whereby the G\"artner-Ellis condition need not be satisfied. Of course, our results are more restrictive in the sense that we only study i.i.d.\ Gaussian random vectors, but we also consider large-scale asymptotics that are not under consideration there.

Our results are also closely related to the important series of papers by Hashorva and H{\"u}sler~\cite{Has2005,HasHu2002,HasHu2003} generalizing the classic Mills ratio Gaussian tail bound~\cite{Sa1962}. We observe that the quadratic program logarithmic asymptote derived in Lemma 1 is also implied by the tail bound dervied in~\cite{Has2005,HasHu2003}. In~\cite{HasHu2003}, the authors derive exact asymptotics for integrals of Gaussian random vectors, and in particular focus on the ``at least one in the set" extremum for half-space extreme value sets. Our proof does not rely on the bound in \cite{Has2005,HasHu2003}.

It would be interesting to strengthen Theorem \ref{thm-main-small} to sharp asymptotics as is performed for large scales in Theorem \ref{thm-sharp-order}. This is hard, since various error terms that can easily be dealt with in the proof of Theorem \ref{thm-sharp-order} as they are much smaller than the leading order, will only become marginally smaller. In would also be interesting to extend our analysis to other multivariate random vectors with non-trivial dependence.

\paragraph{\underline{Notations and Setting}}
 All vector relations should be understood component-wise. Thus, $x > y$ implies that $x^{\sss (j)} > y^{\sss(j)}$ for every component $j$. Following~\cite{KoMa2015} we define a \textit{restricted} large deviation principle (RLDP) as follows: for any $q \in \bbR^d$ a sequence of multivariate $\bbR^d$-valued random variables $\{W_n\}$ satisfies a RLDP with rate function $J \colon \bbR^d \to [0,\infty]$ if 
	\begin{align}
	\lim_{n\to\infty} \frac{1}{v_n} \log \bbP(W_n/a_n > q) = -J(q),
	\end{align}
where $v_n,a_n \to \infty$ as $n\to\infty$. This asymptotic is not a full-fledged large deviations principle (LDP) since it does not provide any insight into what happens for negative $q$, i.e., it only deals with attaining {\em large positive} values. Furthermore, as noted in~\cite{KoMa2015} and \cite{DuLeSu2003}, if $W_n$ satisfies a LDP with  continuous rate function, then it automatically satisfies the RLDP. On the other hand, if the rate function is discontinuous, then it might not satisfy the RLDP.

We can write
	\eqn{
	(X^{\sss(1)}_i, X^{\sss(2)}_i)\stackrel{d}{=} (\sigma^{\sss(1)}Z^{\sss\rm(1)}_i + \mu^{\sss(1)}, \sigma^{\sss(2)}Z^{\sss\rm(2)}_i + \mu^{\sss(2)}),
	}
where $(\sigma^{\sss(1)},\mu^{\sss(1)})$ and $(\sigma^{\sss(2)},\mu^{\sss(2)})$ are the standard deviation and mean of $X^{\sss(1)}_i$ and  $X^{\sss(2)}_i$, respectively, while $(Z^{\sss(1)}_i, Z^{\sss(2)}_i)$ are standard bi-variate normals with correlation coefficient $\rho\in[-1,1]$. Assume that $\mu^{\sss(i)} = 0$ for $i \in \{1,2\}$, without loss of generality.

\section{Right Scale Asymptote}
 We start by analyzing extreme events for bivariate Gaussian random variables:

\begin{lemma}[Extreme events for single normal random variables]
\label{lem:1}
Let $\{a_n\}_{n\geq 1}$ be any unbounded increasing sequence in $n \in \bbN$. Then
  	\begin{align}
    	\label{eq:4}
    	\lim_{n\to\infty} \frac{1}{a_n^2} \log \bbP\left( X_1 > a_n \e
    	\right) = -\essinf_{x > \e} \frac{1}{2} x^T \Sigma^{-1} x.
 	 \end{align}
	\end{lemma}
\begin{proof}
By definition, and with $C$ an  explicit constant,
	\eqan{
  	\label{eq:5}
  	\frac{1}{a_n^2}\log\bbP\left( X_1 > a_n \e \right) &= 
	 \frac{1}{a_n^2}\log \left(C \int_{x > a_n \e} \exp\left(-\frac{1}{2} x^T\Sigma^{-1}x\right) dx \right)\\
	 &=\frac{1}{a_n^2}\log \left( a_n C \int_{x > \e} \exp\left(-a_n^2\frac{1}{2} x^T\Sigma^{-1}x\right) dx \right),\nn
	}
where the second equality follows by a substitution of variables. Laplace's principle~\cite[Chapter 4]{DeZe2010} implies the claim.
\end{proof}

Next, we consider the asymptotics of the logarithmic likelihood of the event $\{\exists i \leq n\colon X_i > a_n u\}$. Note that this is {\it not} the component-wise maximum.

\begin{proposition}[A single index attains the maximum]
\label{prop:1}
Let $a_n := \sqrt{\log n}$. The bivariate Gaussian ensemble satisfies the RLDP limit
  	\begin{align}
    	\label{eq:2}
    	\lim_{n\to\infty} \frac{1}{a_n^2} \log \bbP(\exists i \leq n \colon X_i > a_n u) = 1    - \essinf_{x > u} \frac{1}{2} x^T \Sigma^{-1} x,
  	\end{align}
for $u := (u^{\sss(1)},u^{\sss(2)}) > \sqrt{2}(\s^{\sss(1)},\s^{\sss(2)})$.
\end{proposition}

\begin{remark}[Condition $(u^{\sss(1)},u^{\sss(2)}) > \sqrt{2}(\s^{\sss(1)},\s^{\sss(2)})$]
\label{rem-cond-u-s}
The condition $(u^{\sss(1)},u^{\sss(2)}) > \sqrt{2}(\s^{\sss(1)},\s^{\sss(2)})$ is very natural. Indeed, since the marginal distributions of each of the coordinates is normal with mean zero and standard deviation $\s^{\sss(j)}$, we have that $\max_{i=1}^n X_i^{\sss(j)}/\sqrt{\log{n}} \convas \sqrt{2} \s^{\sss(j)}$ for $j=1,2$. Thus, when $u^{\sss(j)}\leq \sqrt{2}\s^{\sss(j)}$, it is natural to assume that this event does not contribute to the asymptotics in Proposition \ref{prop:1}. In particular, when $(u^{\sss(1)},u^{\sss(2)}) \leq  \sqrt{2}(\s^{\sss(1)},\s^{\sss(2)})$, the limit in \eqref{eq:2} equals zero, whereas the right-hand side is strictly positive.
\end{remark}

\begin{proof}
Observe that
\color{black}
	\eqan{
  	\bbP(\exists i \leq n \colon X_i > a_n u) &= \bbP\left( \cup_{i=1}^n \left\{X_i > a_n\right\}\right)\nn\\
  	&= 1 - \bbP(\cap_{i=1}^n \left\{ X_i > a_n u \right\}^c)\nn\\
  	&=(1-\bbP(\{X_1> a_n u\}^c))\sum_{i=0}^{n-1} \bbP(\{X_i > a_n u\}^c)\nn\\
  	&= \bbP(X_1 > a_n u) \sum_{i=0}^{n-1} b_n^i,	 \label{eq:1}
	}
where $b_n := \bbP(\{X_1 > a_n u\}^c)$. 
\color{black}

From~\eqref{eq:1}, it follows that
  	\begin{align}
    	\label{eq:3}
    	\log \bbP(\exists i \leq n\colon X_i > a_n u) \leq \log \bbP(X_1 > a_n u) + \log n,
  	\end{align}
using the fact that $b_n < 1$ for all finite $n$. Lemma~\ref{lem:1} implies that
  	\begin{align}
    	\label{eq:6}
    	\limsup_{n\to\infty} \frac{1}{a_n^2} \log \bbP\left(\bar X_n > a_n
   	 u\right) \leq 1 - \frac{1}{2}\essinf_{x > u} x^T \Sigma^{-1} x.
 	 \end{align}

Next, for the lower bound, we work with the term $\log \sum_{i=0}^{n-1} b_n^i$ to obtain a finer analysis. In particular,
suppose we demonstrate that, since $b_n \in [0,1]$, $\log(n b_n^{n-1}) \geq \log n + o(\log n)$ as $n \to
\infty$; then, it follows that
	\begin{align}
  	\label{eq:7}
  	\log \sum_{i=0}^{n-1} b_n^i > \log (n b_n^{n-1}) \geq \log
  	n + o(\log n)~\text{as}~n\to\infty.
	\end{align}
Consequently, Lemma~\ref{lem:1}, combined with this result, implies that
	\begin{align}
  	\label{eq:8}
  	\liminf_{n\to\infty}\frac{1}{a_n^2} \log \bbP\left( \exists i \leq n \colon X_i > a_n u\right) \geq
  	1 - \essinf_{x > u} \frac{1}{2} x^T
  	\Sigma^{-1} x,
	\end{align}
thereby completing the proof of the proposition.

It remains to show~\eqref{eq:7}. Observe that the inclusion-exclusion formula implies that 
	\color{black}
	\eqn{
	b_n = \bbP(\{X_1 > a_n u\}^c)=\bbP\big(\{X_1^{\sss(1)} \leq a_n u^{\sss(1)}\}\cup \{X_1^{\sss(2)} \leq a_n u^{\sss(2)}\}\big)
	}
 satisfies
	\begin{align}
	b_n &= \bbP({X_1^{\sss(1)} \leq a_n u^{\sss(1)}}) + \bbP(X_n^{\sss(2)} \leq a_n u^{\sss(2)}) - \bbP(\{X_n^{\sss(1)} \leq a_n u^{\sss(1)}\}\cap \{X_n^{\sss(2)} \leq a_n u^{\sss(2)}\})\\
	&= 2 - \bbP({X_1^{\sss(1)} >  a_n u^{\sss(1)}}) - \bbP(X_n^{\sss(2)} > a_n u^{\sss(2)}) - \bbP(\{X_n^{\sss(1)} \leq a_n u^{\sss(1)}\}\cap \{X_n^{\sss(2)} \leq a_n u^{\sss(2)}\}).\nn
	\end{align}
	\color{black}
Therefore
	\color{black}
	\eqn{
	b_n \geq 1 - \bbP({X_1^{\sss(1)} >  a_n u^{\sss(1)}}) - \bbP(X_n^{\sss(2)} > a_n u^{\sss(2)}).
	}
	\color{black}
By the Taylor series expansion of $\log (1-x) = -x + o(x)$, as well as $1-b_n=o(1)$,
	\color{black}
	\begin{align}
	\log(n b_n^{n-1}) &= \log n + (n-1) \log{b_n}=\log n - (n-1) (1-b_n) + o(\log n)\\ 
	&\geq \log n +n\left(- \bbP({X_1^{\sss(1)} >  a_n u^{\sss(1)}}) - \bbP(X_n^{\sss(2)} > a_n u^{\sss(2)})\right)(1 + o(1)).\nn
	\end{align}
	\color{black}
Next, the Gaussian upper tail bound implies that for large $n$,
	\begin{align}
	\log(n b_n^{n-1}) &\geq \log n - \frac{n}{\sqrt{2\pi}} \frac{1}{\sqrt{\log n}} \sum_{j\in\{1,2\}} \frac{1}{n^{1/2(u^{\sss(i)}/\sigma^{\sss(i)})^2}} \frac{1}{\sigma^{\sss(j)}u^{\sss(j)}}(1 + o(1))\nn\\
	&\geq  \log n - \frac{1}{\sqrt{2\pi}} \frac{2}{\sqrt{\log n}} \max \left\{ \frac{1}{n^{\sss(u^{\sss(1)}/\sigma^{\sss(1)})^2/2-1}} \frac{1}{\sigma^{\sss(1)}u^{\sss(1)}}, \frac{1}{n^{\sss(u^{\sss(2)}/\sigma^{\sss(2)})^2/2-1}} 	
	\frac{1}{\sigma^{\sss(2)}u^{\sss(2)}}  \right\}(1 + o(1)).
	\end{align}
Since $(u^{\sss(1)},u^{\sss(2)}) > \sqrt{2}(\s^{\sss(1)},\s^{\sss(2)})$, it follows that
	\eqn{
	\frac{\log(n b_n^{n-1})}{\log n} \geq 1 + o(1) ~\text{as}~n\to\infty,
	}
thereby completing the proof.
\end{proof}

\begin{lemma}[Analysis of variational problem]
\label{rem-VP}
	By a straightforward calculation,
	\eqan{
	\label{sol-VP}
	J_1(u/\sigma)&:=1-\tfrac{1}{2}\essinf_{x > u} x^T \Sigma^{-1} x
	=\begin{cases}
				1-\frac{1}{2}(u^{\sss(1)}/\sigma^{\sss(1)})^2 &\text{when }u^{\sss(2)}/\sigma^{\sss(2)}\leq \rho u^{\sss(1)}/\sigma^{\sss(1)},\\
				1-\frac{1}{2}(u^{\sss(2)}/\sigma^{\sss(2)})^2 &\text{when }u^{\sss(1)}/\sigma^{\sss(1)}\leq \rho u^{\sss(2)}/\sigma^{\sss(2)},\\
	1-\frac{(u^{\sss(1)}/\sigma^{\sss(1)})^2-2\rho (u^{\sss(1)}/\sigma^{\sss(1)}) (u^{\sss(2)}/\sigma^{\sss(2)})+(u^{\sss(2)}/\sigma^{\sss(2)})^2}{2(1-\rho^2)}
	&\text{otherwise.}
		\end{cases}\nn
	}
\end{lemma}

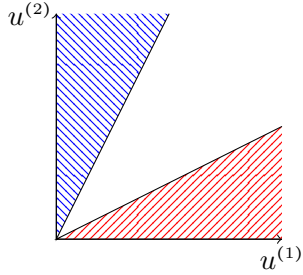
\begin{figure}[h]
	\centering
	\begin{tikzpicture}
		\draw[<->] (3,0)  -- (0,0) -- (0,3); \node at (3.0, -0.25) {$u^{\sss(1)}$}; \node at (-0.35,3.0) {$u^{\sss(2)}$}; 
		\path[draw, pattern=north west lines, pattern color = blue] (1.5,3.0) -- (0,0) -- (0,3.0);
		\path[draw, pattern=north east lines, pattern color = red] (3.0,1.5) -- (0,0) -- (3.0, 0);
	\end{tikzpicture}
	\caption{Fix $\rho=0.5$. The blue cone represents the region where $u^{\sss(1)} \leq \rho u^{\sss(2)}$ and the red cone where $u^{\sss(2)} \leq \rho u^{\sss(1)}$.}
	\label{fig:1}
\end{figure}

\begin{proof} Fix $\rho \in (0,1]$ and without loss of generality assume that $\sigma^{\sss(1)} = \sigma^{\sss(2)} = 1$. We can divide the positive quadrant into three regions as shown in Figure~\ref{fig:1}, where $\rho = 0.5$. Suppose that $u$ is such that $u^{\sss(2)} \leq \rho u^{\sss(1)}$ (see the red region in Figure~\ref{fig:1}), then
	\begin{align}
		\frac{(u^{\sss(1)}/\sigma^{\sss(1)})^2-2\rho (u^{\sss(1)}/\sigma^{\sss(1)}) (u^{\sss(2)}/\sigma^{\sss(2)})+(u^{\sss(2)}/\sigma^{\sss(2)})^2}{2(1-\rho^2)} 
		\leq \tfrac{1}{2}(u^{\sss(1)})^2 < \tfrac{1}{2}(x^{\sss(1)})^2,
	\end{align}
where the final inequality holds for any $x > u$. It follows that $\tfrac{1}{2}\essinf_{x > u} x^T \Sigma^{-1} x  = \tfrac{1}{2}(u^{\sss(1)})^2$. A similar argument shows that $\tfrac{1}{2}\essinf_{x > u} x^T \Sigma^{-1} x = \frac{1}{2}(u^{\sss(2)})^2$ when $u^{\sss(1)} \leq \rho u^{\sss(2)}$. Finally, in the region where neither of these conditions holds (the blank region in Figure~\ref{fig:1}), it is straightforward to verify the Karush-Kuhn-Tucker (KKT) conditions for $u$ and, since $\Sigma^{-1}$ is positive definite, $u$ is the unique optimizer.
\end{proof}

%

As a consequence, we obtain the main result of this section:
\begin{theorem}[Extreme value asymptotics for bi-variate Gaussians]
\label{thm-main-small}
Under the conditions of Proposition~\ref{prop:1}, 
	\eqn{
	\label{20}
	\lim_{n\to\infty} \frac{1}{a_n^2} \log \bbP\left( \bar X_n > a_n u\right) = J(u/\sigma),
	}
where
	\begin{align*}
	\label{J(u)-def}
		J(u)= \begin{cases}
 		1 - \frac{1}{2} \left( {u^{\sss(1)}} \right)^2 &\text{when}~u^{\sss(2)} \leq \rho u^{\sss(1)},\\
		1 - \frac{1}{2} \left( {u^{\sss(2)}} \right)^2 &\text{when}~u^{\sss(1)} \leq \rho u^{\sss(2)},\\
 		\max\big\{2 - \tfrac{1}{2} \left\|{u} \right\|_2^2, 1-\frac{(u^{\sss(1)}/\sigma^{\sss(1)})^2-2\rho (u^{\sss(1)}/\sigma^{\sss(1)}) (u^{\sss(2)}/\sigma^{\sss(2)})+(u^{\sss(2)}/\sigma^{\sss(2)})^2}{2(1-\rho^2)}\big\}&\text{otherwise.}
 	\end{cases}
	\end{align*}
Further, with $(I^\star, J^\star)$ the indices that maximize $\bar X_n$ (i.e., $\bar X_n=(X^{\sss(1)}_{I^\star}, X^{\sss(2)}_{J^\star})$),
	\eqn{
	\label{asymp-indices-right}
	\lim_{n\rightarrow \infty}\prob\big(I^\star\neq J^\star\mid \bar X_n> a_n u\big)=
	\begin{cases}
	1 &\text{ when }2-\tfrac{1}{2}\|u/\sigma\|^2_2>J_1(u/\sigma),\\
	0 &\text{ when }2-\tfrac{1}{2}\|u/\sigma\|^2_2<J_1(u/\sigma).
	\end{cases}
	}
\end{theorem}
\medskip

It is an interesting problem to extend \eqref{asymp-indices-right} to the case where 
	\eqn{
	J(u/\sigma)=2-\tfrac{1}{2}\|u/\sigma\|^2_2=1-\frac{(u^{\sss(1)}/\sigma^{\sss(1)})^2-2\rho (u^{\sss(1)}/\sigma^{\sss(1)}) (u^{\sss(2)}/\sigma^{\sss(2)})+(u^{\sss(2)}/\sigma^{\sss(2)})^2}{2(1-\rho^2)},
	}
but this seems quite difficult, as it requires a finer analysis beyond the principle of the largest term~\cite[Lemma 1.2.15]{DemZei98}.

\begin{proof}
Note that
	\eqan{
	\bbP\left( \bar X_n > a_n u\right)
	&=\bbP(\exists i \leq n\colon X_i > a_n u)\\
	&\quad+\bbP(\exists i\neq j  \leq n\colon X_i^{\sss(1)} > a_n u^{\sss (1)}, X_j^{\sss(2)} > a_n u^{\sss (1)}, \not \exists i \leq n\colon X_i > a_n u).\nn
	}
Depending on $u$, the first or the second term will be dominant. Ignoring the event that $\not \exists i \leq n\colon X_i > a_n u$ (which leads to an upper bound, but this event can also be incorporated in more detail), we have (using that $X_i^{\sss(1)}, X_j^{\sss(2)}$ are Gaussian)
	\eqan{
	\bbP(\exists i\neq j  \leq n\colon X_i^{\sss(1)} > a_n u^{\sss (1)}, X_j^{\sss(2)} > a_n u^{\sss (1)})
	&\approx n(n-1) \bbP(X_i^{\sss(1)} > a_n u^{\sss (1)})\bbP(X_j^{\sss(2)} > a_n u^{\sss (2)})\\
	&\approx n^2 \exp{\Big\{-a_n^2 [(u^{\sss (1)}/\sigma^{\sss(1)})^2+(u^{\sss (2)}/\sigma^{\sss(2)})^2]/2\Big\}}.\nn
	}
Taking log's and dividing by $\log{n}=a_n^2$ gives 
	\eqan{
	\lim_{n\to\infty} \frac{1}{a_n^2} \log\bbP(\exists i\neq j  \leq n\colon X_i^{\sss(1)} > a_n u^{\sss (1)}, X_j^{\sss(2)} > a_n u^{\sss (1)})
	&=2-\tfrac{1}{2}[(u^{\sss (1)}/\sigma^{\sss(1)})^2+(u^{\sss (2)}/\sigma^{\sss(2)})^2]=2 - \tfrac{1}{2} \|u/\sigma\|_2^2\\
	&:=J_2(u/\sigma).\nn
	}

%

Then, by the principle of the largest term~\cite[Lemma 1.2.15]{DemZei98}, it follows that
	\eqan{
	\label{25}
	\lim_{n\to\infty} 
	\frac{1}{a_n^2} \log \bbP(\bar X_n > a_n u) &= \max\Big\{\lim_{n\to\infty} \frac{1}{a_n^2} \log
	\bbP\left( \exists i \leq n\colon X_i > a_n u\right),\\
	&\qquad\qquad\lim_{n\to\infty} \frac{1}{a_n^2} \log \bbP\left(\exists i\neq j  \leq n\colon X_i^{\sss(1)} > a_n u^{\sss (1)}, X_j^{\sss(2)} > a_n u^{\sss (1)} \right) \Big\}\nonumber\\
	\nonumber
	&= \max\left\{J_1(u/\sigma),~J_2(u/\sigma)\right\}.
	}
Further,
	\eqan{
	\lim_{n\rightarrow \infty}\frac{1}{a_n^2} \log\prob\big(I^\star\neq J^\star\mid \bar X_n> a_n u\big)
	&=\lim_{n\rightarrow \infty}\frac{1}{a_n^2} \Big[\log\prob\big(I^\star\neq J^\star,\bar X_n> a_n u\big)-\log\prob\big(\bar X_n> a_n u\big)\Big]\\
	&\leq J_2(u/\sigma)-J(u/\sigma)<0,\nn
	}
when $J_2(u/\sigma)<J(u/\sigma)$, showing that $\prob\big(I^\star\neq J^\star\mid \bar X_n> a_n u\big)=o(1)$ when $J_2(u/\sigma)<J(u/\sigma)$.
Similarly,
	\eqan{
	\lim_{n\rightarrow \infty}\frac{1}{a_n^2} \log\prob\big(I^\star=J^\star\mid \bar X_n> a_n u\big)
	&=\lim_{n\rightarrow \infty}\frac{1}{a_n^2} \Big[\log\prob\big(I^\star=J^\star,\bar X_n> a_n u\big)-\log\prob\big(\bar X_n> a_n u\big)\Big]\\
	&\leq J_1(u/\sigma)-J(u/\sigma)<0,\nn
	}
when $J_1(u/\sigma)<J(u/\sigma)$, showing that $\prob\big(I^\star=J^\star\mid \bar X_n> a_n u\big)=o(1)$ when $J_1(u/\sigma)<J(u/\sigma)$. This proves \eqref{asymp-indices-right} subject to \eqref{20}.

We are left to prove the claim in \eqref{20}, for which we consider several cases:

\noindent {\bf Case (1): $u^{\sss(2)}/\sigma^{\sss(2)} \leq \rho u^{\sss(1)}/\sigma^{\sss(1)}$}. Under the assumption that $u^{\sss (j)} > \sqrt{2} \sigma^{\sss (j)}$ for $j = 1,2$, it is straightforward to see that $1 - \frac{1}{2} \left(u^{\sss(1)}/\sigma^{\sss(1)}\right)^2 > 2 - \frac{1}{2} \|u/\sigma\|_2^2$. Lemma~\ref{rem-VP} implies that
	\eqn{
	\label{26}
	\lim_{n\to\infty} \frac{1}{a_n^2} \log \bbP(\bar X_n > a_n u) = 1 - \tfrac{1}{2} \left(u^{\sss(1)}/\sigma^{\sss(1)}\right)^2=J_1(u/\sigma).
	}

\noindent{\bf Case (2): $u^{\sss(2)}/\sigma^{\sss(2)} \leq \rho u^{\sss(1)}/\sigma^{\sss(1)}$}. The proof follows case (1) and is omitted.

\noindent{\bf Case (3): $u^{\sss(2)}/\sigma^{\sss(2)} > \rho u^{\sss(1)}/\sigma^{\sss(1)}$ and $u^{\sss(1)}/\sigma^{\sss(1)} > \rho u^{\sss(2)}/\sigma^{\sss(2)}$}. Lemma~\ref{rem-VP} implies that~\eqref{25} is simply
	$$
	\max\{J_1(u/\sigma), J_2(u/\sigma)\}
	=\max\Big\{2 - \tfrac{1}{2} \left\|{u} \right\|_2^2, 1-\frac{(u^{\sss(1)}/\sigma^{\sss(1)})^2-2\rho (u^{\sss(1)}/\sigma^{\sss(1)}) (u^{\sss(2)}/\sigma^{\sss(2)})+(u^{\sss(2)}/\sigma^{\sss(2)})^2}{2(1-\rho^2)}\Big\}.
	$$
Which of the two is the maximizer depends sensitively on the relation between $\rho$ and $u$.
\end{proof}

\section{Large-scale Asymptote}
\label{sec-incl-excl}
  We look at $\prob(\bar X_n>a_n u)$, where $a_n\gg \sqrt{\log{n}}$, so that we are considering a large-deviation event. Recall that $(X_i^{\sss(1)},X^{\sss(2)}) = (\sigma^{\sss(1)} Z_i^{\sss(1)}, \sigma^{\sss(2)} Z_i^{\sss(2)})$. By changing $u=(u^{\sss(1)}, u^{\sss(2)})$ if needed, it thus suffices to study the standard case, and in what follows, we will therefore focus on $\prob(\bar Z_n>a_n u)$. We prove the following two main theorems:

\begin{theorem}[Leading order asymptotics extremes]
\label{thm-leading-order}
For any $a_n\gg \sqrt{\log{n}}$, and with $u^{\sss(2)}\leq u^{\sss(1)}$,
	\eqn{
	\label{mod-LD-BVN}
	 \lim_{n\to\infty} \frac{1}{a_n^2} \log \prob\big(\bar Z_n> a_n u) = -I(u),
	}
where 
	\eqn{
	\label{rate-BVN}
	I(u)=\begin{cases}
	\frac{1}{2}(u^{\sss(1)})^2 &\text{when }u^{\sss(2)}\leq \rho u^{\sss(1)},\\
	\frac{1}{2}\min\big\{\|u\|^2_2, \frac{(u^{\sss(1)})^2-2\rho u^{\sss(1)} u^{\sss(2)}+(u^{\sss(2)})^2}{1-\rho^2}\big\}
	&\text{otherwise.}
	\end{cases}
	}
\end{theorem}
\medskip

Theorem \ref{thm-leading-order} is a special example of the two-dimensional Cram\'er Theorem (see e.g., \cite{DemZei98} or \cite{Holl00}). The main aim of this paper in the large-scale asymptotics is to study the sharp asymptotics of a large deviation of the bi-variate normal distribution, which is quite interesting indeed:

\begin{theorem}[Sharp asymptotics extremes]
\label{thm-sharp-order}
For any $a_n\gg \sqrt{\log{n}}$, and with $u^{\sss(2)}\leq u^{\sss(1)}$,
	\eqn{
	\label{mod-LD-BVN-sharp}
	\lim_{n\rightarrow \infty} a_n^bn^c\ex^{a_n^2 I(u)}\prob\big(\bar Z_n> a_n u) = K,
	}
where 
	\eqn{
	\label{abc-BVN}
	b=1+\indic{u^{\sss(2)}\geq \rho u^{\sss(1)}},
	\qquad
	c=1+\indic{I(u)=\|u\|^2/2},
	}
and 
	\eqn{
	K=
	\begin{cases}
	\frac{1}{2\pi u^{\sss(1)}u^{\sss(2)}} 	&\text{ when }I(u)=\tfrac{1}{2}\|u\|^2, u^{\sss(1)}\neq u^{\sss(2)},\\
	\frac{1}{4\pi u^{\sss(1)}u^{\sss(2)}} 	&\text{ when }I(u)=\tfrac{1}{2}\|u\|^2, u^{\sss(1)} = u^{\sss(2)},\\
	\frac{1}{2\pi u^{\sss(1)}}				&\text{ when }I(u)<\tfrac{1}{2}\|u\|^2, u^{\sss(2)}<\rho u^{\sss(1)},\\
	\frac{1}{4\pi u^{\sss(1)}}				&\text{ when }I(u)<\tfrac{1}{2}\|u\|^2, u^{\sss(2)}=\rho u^{\sss(1)},\\
	\frac{1-\rho^2}{2\pi(u^{\sss(1)}-\rho u^{\sss(2)})(u^{\sss(2)}-\rho u^{\sss(1)})}&\text{ when }I(u)<\tfrac{1}{2}\|u\|^2, u^{\sss(2)}>\rho u^{\sss(1)}.
	\end{cases}
	}
Consequently, with $(I^\star, J^\star)$ the indices that maximize $\bar Z_n$ (i.e., $\bar Z_n=(Z^{\sss(1)}_{I^\star}, Z^{\sss(2)}_{J^\star})$),
	\eqn{
	\lim_{n\rightarrow \infty}\prob\big(I^\star\neq J^\star\mid \bar Z_n> a_n u\big)=
	\begin{cases}
	1 &\text{ when }I(u)=\tfrac{1}{2}\|u\|^2_2,\\
	0 &\text{ otherwise}.
	\end{cases}
	}
\end{theorem}
\medskip

\proof We prove Theorem \ref{thm-leading-order} and \ref{thm-sharp-order} in one go. We note that
	\eqn{
	\prob(\bar Z_n>a_n u)=\prob\Big(\bigcup_{(i,j)} \{Z^{\sss\rm(1)}_i>a_n u^{\sss(1)}, Z^{\sss\rm(2)}_j> a_n u^{\sss(2)}\}\Big).
	}
We obtain
	\eqn{
	\prob(\bar Z_n>a_n u)=\prob\Big(\bigcup_{(i,j)} A_{(i,j)}\Big),
	}
where 
	\eqn{
	A_{(i,j)}=\{Z^{\sss\rm(1)}_i>a_n u^{\sss(1)}, Z^{\sss\rm(2)}_j> a_n u^{\sss(2)}\}.
	}
From this formula, we see the importance of symmetry, as $A_{(i,j)}=A_{(j,i)}$ for the symmetric case where $u^{\sss(1)}=u^{\sss(2)}$, but not when this is not the case. 
In the symmetric case where $u^{\sss(1)}=u^{\sss(2)}$, we note that $A_{(i,j)}=A_{(j,i)}$, so we may write, instead,
	\eqn{
	\prob(\bar Z_n>a_n u)=\prob\Big(\bigcup_{(i,j)\colon i\leq j} \{Z^{\sss\rm(1)}_i>a_n u^{\sss(1)}, Z^{\sss\rm(2)}_j> a_n u^{\sss(2)}\}\Big).
	}

\paragraph{\bf Using inclusion-exclusion.} We use inclusion-exclusion to obtain that
	\eqn{
	\sum_{(i,j)} \prob\big(A_{(i,j)}\big)
	-e_n(u)\leq \prob(\bar Z_n>a_n u)\leq \sum_{(i,j)} \prob\big(A_{(i,j)}\big),
	}
where
	\eqn{
	\label{error-term-enu}
	e_n(u)=\frac{1}{2}\sum_{(i,j)\neq (k,l)}\prob\big(A_{(i,j)}\cap A_{(k,l)}\big),
	}
while for the symmetric case, we sum over {\em ordered} pairs $(i,j)$ with $i\leq j$ instead.

Below, we analyse each of these terms. We separate between the case where (i) the indices are different; (ii) they are equal but the probability simplifies; (iii) they are different and we need to perform the integral over the joint density using the Laplace method. We start with the asymmetric case where $u^{\sss(1)}\neq u^{\sss(2)}$, remarking on the extension to the symmetric case at the end of the proof. Without loss of generality, we may assume that $u^{\sss(1)}>u^{\sss(2)}$.

\paragraph{\bf Sum of probabilities: unequal indices.}  First consider the case where $i\neq j$. Then, since $(Z^{\sss\rm(1)}_i, Z^{\sss\rm(2)}_j)$ are i.i.d.\ standard normal random variables, 
	\eqn{
	\sum_{(i,j)\colon i\neq j} \prob\big(A_{(i,j)}\big)
	=n(n-1) \big[1-\Phi(a_n u^{\sss(1)})\big]\big[1-\Phi(a_n u^{\sss(2)})\big],
	}
where $\Phi(x)=\prob(Z\leq x)$ is the error-function or the distribution function of a standard normal. By the asymptotics, for $x$ large,
	\eqn{
	1-\Phi(x)=\frac{1}{\sqrt{2\pi} x}\ex^{-x^2/2}(1+O(x^{-2})),
	}
we thus obtain that
	\eqn{
	\label{sum-probs-equal}
	\sum_{(i,j)\colon i\neq j} \prob\big(A_{(i,j)}\big)
	=n^2 a_n^2 \frac{1}{2\pi u^{\sss(1)}u^{\sss(2)}}\ex^{-a_n^2 \|u\|_2^2/2}(1+o(1)).
	}

\paragraph{\bf Sum of probabilities: simple cases of equal indices.} We next proceed with the case where $i=j$, for which we get	
	\eqn{
	\sum_{(i,i)} \prob\big(A_{(i,i)}\big)
	=n\prob\big(Z^{\sss\rm(1)}>a_n u^{\sss(1)}, Z^{\sss\rm(2)}> a_n u^{\sss(2)}\big).
	}
We use that, conditionally on $Z^{\sss\rm(1)}$, the law of $Z^{\sss\rm(2)}$ equals $Z^{\sss\rm(2)}=\rho Z^{\sss\rm(1)} +\sqrt{1-\rho^2}Z$, where $Z$ is independent of $Z^{\sss\rm(1)}$. We thus get that
	\eqan{
	\prob\big(Z^{\sss\rm(1)}>a_n u^{\sss(1)}, Z^{\sss\rm(2)}> a_n u^{\sss(2)}\big)
	&=\prob\big(Z^{\sss\rm(1)}>a_n u^{\sss(1)}, \rho Z^{\sss\rm(1)} +\sqrt{1-\rho^2}Z> a_n u^{\sss(2)}\big)\nn\\
	&=\expec\Big[\indic{Z^{\sss\rm(1)}>a_n u^{\sss(1)}}\prob\big(\sqrt{1-\rho^2} Z>a_n u^{\sss(2)}-\rho Z^{\sss\rm(1)}\mid Z^{\sss\rm(1)} \big)\Big].
	}
When $\rho u^{\sss(1)}>u^{\sss(2)}$,
	\eqn{
	\prob\big(\sqrt{1-\rho^2} Z>a_n u^{\sss(2)}-\rho Z^{\sss\rm(1)}\mid Z^{\sss\rm(1)} \big)=1+o(1),
	}
so that 
	\eqn{
	\label{final-easy-1}
	\sum_{(i,i)} \prob\big(A_{(i,i)}\big)
	=n\prob\big(Z^{\sss\rm(1)}>a_n u^{\sss(1)}\big)(1+o(1))
	=na_n \frac{1}{2\pi u^{\sss(1)}}\ex^{-a_n^2 (u^{\sss\rm(1)})^2/2}(1+o(1)).
	}
When $\rho u^{\sss(1)}=u^{\sss(2)}$, instead
	\eqn{
	\prob\big(\sqrt{1-\rho^2} Z>a_n u^{\sss(2)}-\rho Z^{\sss\rm(1)}\mid Z^{\sss\rm(1)} \big)=\frac{1}{2}+o(1),
	}
since $Z^{\sss\rm(1)}-a_n u^{\sss(1)}=\op(1)$ when $Z^{\sss\rm(1)}>a_n u^{\sss(1)}$. Thus, for $\rho u^{\sss(1)}=u^{\sss(2)}$ this leads to 
	\eqn{
	\label{final-easy-2}
	\sum_{(i,i)} \prob\big(A_{(i,i)}\big)
	=n\prob\big(Z^{\sss\rm(1)}>a_n u^{\sss(1)}\big)(1+o(1))
	=na_n \frac{1}{4\pi u^{\sss(1)}}\ex^{-a_n^2 (u^{\sss\rm(1)})^2/2}(1+o(1)).
	}

\paragraph{\bf Sum of probabilities: Lapace integral for equal indices.} When the above simple cases do not apply, we write $\prob\big(Z^{\sss\rm(1)}>a_n u^{\sss(1)}, Z^{\sss\rm(2)}> a_n u^{\sss(2)}\big)$ explicitly as a two-dimensional integral as
	\eqan{
	\prob\big(Z^{\sss\rm(1)}>a_n u^{\sss(1)}, Z^{\sss\rm(2)}> a_n u^{\sss(2)}\big)
	=\frac{1}{2\pi}\int_{a_n u^{\sss(1)}}^{\infty} \int_{a_n u^{\sss(2)}}^{\infty}  \ex^{-(x_1^2 -2\rho x_1x_2 +x_2^2)/2(1-\rho^2)}dx_2dx_1.
	}
We rescale the integrands by $a_n$ to obtain
	\eqan{
	\prob\big(Z^{\sss\rm(1)}>a_n u^{\sss(1)}, Z^{\sss\rm(2)}> a_n u^{\sss(2)}\big)
	=\frac{a_n^2}{2\pi}\int_{u^{\sss(1)}}^{\infty} \int_{u^{\sss(2)}}^{\infty}  \ex^{-a_n^2(x_1^2 -2\rho x_1x_2 +x_2^2)/2(1-\rho^2)}dx_2dx_1.
	}
This is a classic example of a Laplace integral. Thus, the integral is dominated by the minimum of $(x_1^2 -2\rho x_1x_2 +x_2^2)/2(1-\rho^2)$ over all $(x_1,x_2)$ for which $x_1\geq u^{\sss(1)}, x_2\geq u^{\sss(2)}$. Since $(x_1^2 -2\rho x_1x_2 +x_2^2)/2(1-\rho^2)$ is convex, this minimum is attained at one of the boundaries. Since $\rho u^{\sss(1)}<u^{\sss(2)}$, this minimum is attained at $x_1=u^{\sss(1)}, x_2=u^{\sss(2)}$ (see also the analysis in Lemma \ref{lem:1}).

Thus,
	\eqan{
	\prob\big(Z^{\sss\rm(1)}>a_n u^{\sss(1)}, Z^{\sss\rm(2)}> a_n u^{\sss(2)}\big)
	&=\frac{a_n^2}{2\pi}\exp{\big\{-a_n^2 \frac{(u^{\sss(1)})^2-2\rho u^{\sss(1)} u^{\sss(2)}+(u^{\sss(2)})^2}{2(1-\rho^2)}\big\}}\nn\\
	&\times\int_{u^{\sss(1)}}^{\infty} \int_{u^{\sss(2)}}^{\infty}  \exp{\big\{-a_n^2\frac{(x_1^2 -2\rho x_1x_2 +x_2^2)-(u^{\sss(1)})^2+2\rho u^{\sss(1)} u^{\sss(2)}-(u^{\sss(2)})^2}{2(1-\rho^2)}}dx_2dx_1.
	}
Therefore, we obtain that
	\eqan{
	&2\pi a_n^{-2} \exp{\big\{a_n^2 \frac{(u^{\sss(1)})^2-2\rho u^{\sss(1)} u^{\sss(2)}+(u^{\sss(2)})^2}{2(1-\rho^2)}\big\}}\prob\big(Z^{\sss\rm(1)}>a_n u^{\sss(1)}, Z^{\sss\rm(2)}> a_n u^{\sss(2)}\big)\nn\\
	&=\int_{u^{\sss(1)}}^{\infty} \int_{u^{\sss(2)}}^{\infty}  \exp{\big\{-a_n^2\frac{(x_1-u^{\sss(1)})^2 -2\rho (x_1-u^{\sss(1)})(x_2-u^{\sss(2)}) +(x_2-u^{\sss(2)})^2}{2(1-\rho^2)}}\nn\\
	&\quad \times  \exp{\Big\{-a_n^2\frac{\big[2u^{\sss(1)}(x_1-u^{\sss(1)})+2u^{\sss(2)}(x_2-u^{\sss(2)})
	-\rho u^{\sss(1)}(x_2-u^{\sss(2)})-\rho u^{\sss(2)}(x_1-u^{\sss(1)})\big]}{2(1-\rho^2)}\Big\}}dx_2dx_1.
	}
Since $\rho u^{\sss(1)}<u^{\sss(2)}$, we have that the quadratic function inside the exponential is minimized for $x_1=u^{\sss(1)}, x_2=u^{\sss(2)}$. Shifting both integrands by $u^{\sss(1)}$ and $u^{\sss(2)}$ respectively, leads to
	\eqan{
	&2\pi a_n^{-2} \exp{\big\{a_n^2 \frac{(u^{\sss(1)})^2-2\rho u^{\sss(1)} u^{\sss(2)}+(u^{\sss(2)})^2}{2(1-\rho^2)}\big\}}\prob\big(Z^{\sss\rm(1)}>a_n u^{\sss(1)}, Z^{\sss\rm(2)}> a_n u^{\sss(2)}\big)\nn\\
	&=\int_{0}^{\infty} \int_{0}^{\infty}  \exp{\Big\{-a_n^2\frac{x_1^2 -2\rho x_1x_2 +x_2^2}{2(1-\rho^2)}\Big\}}\exp{\Big\{-a_n^2\frac{2u^{\sss(1)}x_1+2u^{\sss(2)}x_2
	-\rho u^{\sss(1)}x_2-\rho u^{\sss(2)}x_1}{2(1-\rho^2)}\Big\}}dx_2dx_1.
	}
Now rescaling {\em both} integrands by $a_n^{-2}$ leads to 
	\eqan{
	&2\pi a_n^{2} \exp{\Big\{a_n^2 \frac{(u^{\sss(1)})^2-2\rho u^{\sss(1)} u^{\sss(2)}+(u^{\sss(2)})^2}{2(1-\rho^2)}\Big\}}\prob\big(Z^{\sss\rm(1)}>a_n u^{\sss(1)}, Z^{\sss\rm(2)}> a_n u^{\sss(2)}\big)\nn\\
	&=\int_{0}^{\infty} \int_{0}^{\infty}  \exp{\Big\{-a_n^{-2}\frac{x_1^2 -2\rho x_1x_2 +x_2^2}{2(1-\rho^2)}\Big\}}
	\exp{\Big\{-\frac{(u^{\sss(1)}-\rho u^{\sss(2)})x_1+(u^{\sss(2)}-\rho u^{\sss(1)})x_2}{1-\rho^2}\Big\}}dx_2dx_1.
	}
Again we see the significance of the assumption that $\rho u^{\sss(1)}<u^{\sss(2)}$, which implies that both linear terms have a {\em negative} coefficient, and thus the exponential functions are integrable. As a result, the first exponential in the integral only leads to an error term, so that 
	\eqan{
	\label{final-uneq-hard}
	&2\pi a_n^{2} \exp{\Big\{a_n^2 \frac{(u^{\sss(1)})^2-2\rho u^{\sss(1)} u^{\sss(2)}+(u^{\sss(2)})^2}{2(1-\rho^2)}\Big\}}\prob\big(Z^{\sss\rm(1)}>a_n u^{\sss(1)}, Z^{\sss\rm(2)}> a_n u^{\sss(2)}\big)\nn\\
	&=(1+o(1))\int_{0}^{\infty} \int_{0}^{\infty}\exp{\Big\{-\frac{(u^{\sss(1)}-\rho u^{\sss(2)})x_1+(u^{\sss(2)}-\rho u^{\sss(1)})x_2}{1-\rho^2}\Big\}}dx_2dx_1\nn\\
	&=(1+o(1)) (1-\rho^2)(u^{\sss(1)}-\rho u^{\sss(2)})^{-1}(u^{\sss(2)}-\rho u^{\sss(1)})^{-1}.
	}
Combining \eqref{final-easy-1}--\eqref{final-easy-2} with \eqref{final-uneq-hard} yields the asymptotics of the sum of probabilities. Note that the final outcome yields \eqref{mod-LD-BVN-sharp} in the asymmetric case, so what is left is to show that the error term $e_n(u)$ is of smaller order.
\medskip

\paragraph{\bf The symmetric case: sum of probabilities.} We now look at the sum of probabilities for the symmetric case, and analyse $\prob(A_{(i,j)})$ there. The analysis for the case where $i\neq j$ is identical to the one above, except for the fact that the prefactor (due to the number of pairs $(i,j)$) is changed from $n(n-1)$ to $n(n-1)/2$. The contribution for the case where $i=j$ is also the same as above. In fact, it is easy to see that for $u^{\sss(1)}=u^{\sss(2)}=u$, we have $I((u,u))=u^2$ when $\rho\leq 0$, while $I((u,u))=u^2/(1+\rho)$ for $\rho>0$, since
	\eqn{
	\frac{(u^{\sss(1)})^2-2\rho u^{\sss(1)} u^{\sss(2)}+(u^{\sss(2)})^2}{1-\rho^2}
	=u^2 \frac{2(1-\rho)}{1-\rho^2}=u^2 \frac{2}{1+\rho}<2u^2,
	}
precisely when $\rho>0$. 

\paragraph{\bf The error term $e_n(u)$: asymmetric case.}  In dealing with error terms, we will make essential use of the fact that $a_n\gg \sqrt{\log{n}}$. This condition implies that if a certain event $A$ satisfies $\prob(A)\leq \ex^{-a_n^2 J}$ for some $J>I(u)$, then $\prob(A)$ will constitute an error term in evaluating $\prob(\bar Z_n>a_n u)$, irrespective of the precise powers of $a_n$ and $n$. Recall \eqref{error-term-enu}. We investigate the different ways that $(i,j)\neq (k,l)$ can occur, depending on the cardinality of $\{i,j,k,l\}$ which ranges from $2$ to $4$. Below, we assume throughout the analysis that the indices $i,j,k,l$ used are {\em distinct}.

\begin{description}

\item[Case (2a): $(i,j), (j,i)$] This corresponds to $\prob(\bar Z_n>a_n \bar{u})$, where $\bar{u}=(u^{\sss(1)}\vee u^{\sss(2)}, u^{\sss(1)}\vee u^{\sss(2)})$ and $x\vee y=\max\{x,y\}$ for $x,y\in {\mathbb R}$. This case was investigated in the previous step, and we see that the rate at speed $a_n^2$ equals $I((\bar{u},\bar{u}))>I(u)$, since $u^{\sss(1)}\neq u^{\sss(2)}$.

\item[Case (2b): $(i,i), (i,j)$ or $(i,i), (j,i)$] By independence, these probabilities equal $\prob\big(Z^{\sss\rm(1)}>a_n u^{\sss(1)}, Z^{\sss\rm(2)}> a_n u^{\sss(2)}\big)\prob(Z>a_nu^{\sss(2)})$ and  $\prob\big(Z^{\sss\rm(1)}>a_n u^{\sss(1)}, Z^{\sss\rm(2)}> a_n u^{\sss(2)}\big)\prob(Z>a_nu^{\sss(1)})$, respectively. Obviously, the rate at speed $a_n^2$ is strictly larger than $I(u)$.

\item[Case (3a): $(i,i), (j,k)$] By independence, this probability equals $\prob(A_{(i,i)})\prob(A_{(j,k)})$, the rate at speed $a_n^2$ again being strictly larger than $I(u)$.

\item[Case (3b): $(i,j), (j,k)$] By independence, this probability equals $\prob(A_{(j,j)})\prob(Z>a_nu^{\sss(1)})\prob(Z>a_nu^{\sss(2)})$, the rate at speed $a_n^2$ again being strictly larger than $I(u)$.

\item[Case (3c): $(i,j), (k,j)$] This case is similar. 

\item[Case (4): $(i,j), (k,l)$] By independence, this probability equals $\prob(A_{(i,j)})\prob(A_{(k,l)})$, the rate at speed $a_n^2$ being at least $2I(u)$, which is again strictly larger than $I(u)$.
\end{description}
Together, these cases show that $e_n(u)$ is of smaller order than the sum of probabilities in the asymmetric case.

\paragraph{\bf The error term $e_n(u)$: symmetric case.} 	This analysis is similar the the asymmmetric case, except that some cases do not arise. We again go through the distinct possibilities, writing $u=u^{\sss(1)}=u^{\sss(2)}$:

\begin{description}

\item[Case (2): $(i,i), (i,j)$ or $(i,i), (j,i)$] By independence, this probability equals $\prob\big(Z^{\sss\rm(1)}>a_n u, Z^{\sss\rm(2)}> a_n u\big)\prob(Z>a_nu)$. Obviously, the rate at speed $a_n^2$ is strictly larger than $I((u,u))$.

\item[Case (3a): $(i,i), (j,k)$] By independence, this probability equals $\prob(A_{(i,i)})\prob(A_{(j,k)})$, the rate at speed $a_n^2$ again being strictly larger than $I((u,u))$.

\item[Case (3b): $(i,j), (j,k)$] By independence, this probability equals $\prob(A_{(j,j)})\prob(Z>a_nu)\prob(Z>a_nu)$, the rate at speed $a_n^2$ again being strictly larger than $I((u,u))$.

\item[Case (4): $(i,j), (k,l)$] By independence, this probability equals $\prob(A_{(i,j)})\prob(A_{(k,l)})$, the rate at speed $a_n^2$ being at least $2I((u,u))$, which is again strictly larger than $I((u,u))$.
\end{description}
Together, these cases show that $e_n(u)$ is also of smaller order than the sum of probabilities in the symmetric case.
\qed	
\bigskip


\paragraph{\bf Acknowledgement.} The work of RvdH is supported by the Netherlands Organisation for Scientific Research (NWO) through VICI grant 639.033.806 and the Gravitation {\sc Networks} grant 024.002.003. The work of HH is partially supported by the National Science Foundation through grants CMMI-1636069 and DMS-1812197.

\bibliographystyle{plain}
\bibliography{refs}

\end{document}